\documentclass[12pt,a4paper]{article}

\usepackage{enumerate}
\usepackage{amsmath,latexsym,amssymb,amsfonts,amsthm}
\usepackage{eucal,color}
\usepackage{pstricks} 
\usepackage{eufrak}
\usepackage[hang,flushmargin]{footmisc}
\newcounter{Scounter}
\theoremstyle{theorem}
\newtheorem{thm}{Theorem}

\newtheorem{cor}[thm]{Corollary}

\newtheorem{lem}[thm]{Lemma}
\newtheorem{obs}[thm]{Observation}
\newtheorem{Q}[thm]{Question} 

\newtheorem{claim}{Claim}
\newtheorem{subclaim}{Subclaim}[claim]

\newtheorem{con}[thm]{Conjecture}

\numberwithin{equation}{section}

\def\P{{ \mathcal{P}}}

\def\Q{{ \mathcal{Q}}}

\def\B{{ \mathcal{B}}}

\bfseries\normalfont

\date{}

\begin{document}

\title{Partitioning a graph into highly connected subgraphs}
\author{Valentin Borozan$^{1,5}$, Michael Ferrara$^2$, Shinya Fujita$^3$ \\Michitaka Furuya$^4$, Yannis Manoussakis$^5$, \\Narayanan N$^{6}$ and Derrick Stolee$^7$}

\maketitle

\noindent\footnotetext[1]{Alfr\'ed R\'enyi Institute of Mathematics, Hungarian Academy of Sciences, H-1053 Budapest, Re\'altanoda u. 13-15, Hungary; valentin.borozan@gmail.com}
\noindent\footnotetext[2]{Department of Mathematical and Statistical Sciences, Univ.\ of Colorado Denver;  michael.ferrara@ucdenver.edu}
\noindent\footnotetext[3]{Corresponding author; International College of Arts and Sciences, Yokohama City University 22-2, Seto, Kanazawa-ku, Yokohama, 236-0027, Japan; shinya.fujita.ph.d@gmail.com}
\noindent\footnotetext[4]{Department of Mathematical Information Science, Tokyo University of Science, 1-3 Kagurazaka, Shinjuku-ku, Tokyo 162-8601, Japan; michitaka.furuya@gmail.com}
\noindent\footnotetext[5]{L.R.I.,B\^at.490, University Paris 11 Sud, 91405 Orsay Cedex, France; yannis.manoussakis@lri.fr}
\noindent\footnotetext[6]{Department of Mathematics, Indian Institute of Technology Madras, Chennai, India.; naru@iitm.ac.in}
\noindent\footnotetext[7]{Department of Mathematics, Department of Computer Science, Iowa State University, Ames, IA. dstolee@iastate.edu}

\begin{abstract}
Given $k\ge 1$, a \emph{$k$-proper partition} of a graph $G$ is a partition ${\mathcal P}$ of $V(G)$ such that each part $P$ of ${\mathcal P}$ induces a $k$-connected subgraph of $G$.  
We prove that if $G$ is a graph of order $n$ such that $\delta(G)\ge \sqrt{n}$, then $G$ has a $2$-proper partition with at most $n/\delta(G)$ parts. 
The bounds on the number of parts and the minimum degree are both best possible.  
We then prove that if $G$ is a graph of order $n$ with minimum degree 
\[
\delta(G)\ge\sqrt{c(k-1)n},
\] 
where $c=\frac{2123}{180}$, then $G$ has a $k$-proper partition into at most $\frac{cn}{\delta(G)}$ parts. 
This improves a result of Ferrara, Magnant and Wenger [Conditions for Families of Disjoint $k$-connected Subgraphs in a Graph, \textit{Discrete Math.} \textbf{313} (2013), 760--764], and both the degree condition and the number of parts is best possible up to the constant $c$.  
\end{abstract}

\section{Introduction}\label{sec1}

A graph $G$ is \textit{$k$-connected} if the removal of any collection of fewer than $k$ vertices from $G$ results in a connected graph with at least two vertices.  In this paper, we are interested in determining minimum degree conditions that ensure that the vertex set of a graph can be partitioned into sets that each induce a $k$-connected subgraph.  In a similar vein, Thomassen \cite{thom} showed that for every $s$ and $t$, there exists a function $f(s,t)$ such that if $G$ is an $f(s,t)$-connected graph, then $V(G)$ can be decomposed into sets $S$ and $T$ such that $S$ induces an $s$-connected subgraph and $T$ induces a $t$-connected subgraph.  In the same paper, Thomassen conjectured that $f(s,t)=s+t+1$, which would be best possible, and Hajnal \cite{H} subsequently showed that $f(s,t)\le 4s+4t-13$. 

From a vulnerability perspective, highly connected graphs represent robust networks that are resistant to multiple node failures.  When a graph is not highly connected, it is useful to partition the vertices of the graph so that every part induces a highly connected subgraph.  For example, Hartuv and Shamir~\cite{HS} designed a clustering algorithm where the vertices of a graph $G$ are partitioned into highly connected induced subgraphs.  It is important in such applications that each part is highly connected, but also that there are not too many parts.

Given a simple graph $G$ and an integer $k \geq 1$, we say a partition $\P$ of $V(G)$ is \emph{$k$-proper} if for every part $P \in \P$, the induced subgraph $G[P]$ is $k$-connected.  Ferrara, Magnant, and Wenger~\cite{FMW} gave a minimum-degree condition on $G$ that guarantees a $k$-proper partition.

\begin{thm}[Ferrara, Magnant, Wenger~\cite{FMW}]
\label{thmA}
Let $k\geq 2$ be an integer, and let $G$ be a graph of order $n$.
If $\delta (G)\geq 2k\sqrt{n}$, then $G$ has a $k$-proper partition $\P$ with $|\P|\leq 2kn/\delta (G)$.
\end{thm}

In addition, they present a graph $G$ with $\delta(G) = (1+o(1))\sqrt{(k-1)n}$ that contains no $k$-proper partition.
This example, which we make more precise below, leads us to make the following conjecture.

\begin{con}
\label{con1}
Let $k\geq 2$ be an integer, and let $G$ be a graph of order $n$.
If $\delta (G)\geq \sqrt{(k-1)n}$, then $G$ has a $k$-proper partition $\P$ with $|\P|\leq \frac{n-k+1}{\delta-k+2}$.
\end{con}

 To see that the degree condition in Conjecture \ref{con1}, if true, is approximately best possible, let $n,\ell$ and $p$ be integers such that $\ell=\sqrt{(k-1)(n-1)}$ and $p= \frac{\ell}{(k-1)}=\frac{n-1}{\ell}$.  Starting from $H=pK_{\ell}$, so that $|H|=n-1$, construct the graph $G$ by adding a new vertex $v$ that is adjacent to exactly $k-1$ vertices in each component of $H$.  Then $\delta(G)=\ell-1$, but there is no $k$-connected subgraph of $G$ that contains $v$.  

To see that the number of components in Conjecture \ref{con1} is best possible, let $r$ and $s$ be integers such that $r=\sqrt{(k-1)n}-k+2$ and $s=\frac{n-k+1}{r}$.  Consider then $G=sK_r\vee K_{k-1}$, which has minimum degree $r+k-2=\sqrt{(k-1)n}$, while every $k$-proper partition has at least $s=\frac{n-k+1}{\delta-k+2}$ parts.  

As an interesting comparison, Nikiforov and Shelp~\cite{NS07} give an approximate version of Conjecture~\ref{con1} with a slightly weaker degree condition.  Specifically, they prove that if $\delta(G) \geq \sqrt{2(k-1)n}$, then there exists a partition of $V(G)$ such that $n - o(n)$ vertices are contained in parts that induce $k$-connected subgraphs.
 
In Section~\ref{sec2}, we verify Conjecture~\ref{con1} in the case $k = 2$.

\begin{thm}
\label{thm1}
Let $G$ be a graph of order $n$.
If $\delta (G)\geq \sqrt{n}$, then $G$ has a $2$-proper partition $\P$ with $|\P|\leq (n-1)/\delta (G)$.  
\end{thm}
 
Ore's Theorem \cite{Ore} states that if $G$ is a graph of order $n\ge 3$ such that $\sigma_2(G) = \min\{ d(u) + d(v) \mid uv \notin E(G)\} \ge n$, then $G$ is hamiltonian, and therefore has a trivial $2$-proper partition.
As demonstrated by Theorem \ref{thm1} however, the corresponding minimum degree threshold is considerably different. 
Note as well that if $G$ has a $2$-factor ${\mathcal F}$, then $G$ has a $2$-proper partition, as each component of ${\mathcal F}$ induces a hamiltonian, and therefore $2$-connected, graph.  Consequently, the problem of determining if $G$ has a 2-proper partition can also be viewed as an extension of the 2-factor problem \cite{AKBook, PSurv}, which is itself one of the most natural generalizations of the hamiltonian problem \cite{RGHamSurv1,RGHamSurv2,RGHamSurv3}.  

In Section~\ref{sec3}, we improve the bound on the minimum degree to guarantee a $k$-proper partition for general $k$, as follows.

\begin{thm}
\label{thm2}
If $G$ is a graph of order $n$ with
\[
\delta(G) \geq \sqrt{\frac{2123}{180}(k-1)n}\]
then $G$ has a $k$-proper partition into at most $\frac{2123n}{180\delta}$ parts.
\end{thm}

Conjecture \ref{con1} yields that both the degree condition and the number of parts in the partition in Theorem \ref{thm2} are best possible up to the constant $\frac{2123}{180}$. Our proof of Theorem~\ref{thm2} has several connections to work of Mader~\cite{Mader} and Yuster~\cite{Yuster}, discussed in Section~\ref{sec3}.
One interesting aspect of our proof is that under the given conditions, the greedy method of building a partition by iteratively removing the largest $k$-connected subgraph will produce a $k$-proper partition.


\subsection*{Definitions and Notation}

All graphs considered in this paper are finite and simple, and we refer the reader to \cite{B} for terminology and notation not defined here.  
Let $H$ be a subgraph of a graph $G$, and for a vertex $x\in V(H)$, let $N_H(x)=\{y\in V(H) \mid xy\in E(H)\}$. 

A subgraph $B$ of a graph $G$ is a \textit{block} if $B$ is either a bridge or a maximal 2-connected subgraph of $G$.  
It is well-known that any connected graph $G$ can be decomposed into blocks.  
A pair of blocks $B_1, B_2$ are necessarily edge-disjoint, and if two blocks intersect, then their intersection is exactly some vertex $v$ that is necessarily a cut-vertex in $G$.  
The {\it block-cut-vertex graph} of $G$ is defined to be the bipartite graph $T$ with one partite set comprised of all cut-vertices of $G$ and the other partite set comprised of all blocks of $G$. 
For a cut-vertex $v$ and a block $B$, $v$ and $B$ are adjacent in $T$ if and only if $v$ is a vertex of $B$ in $G$. 

\section{2-Proper Partitions}\label{sec2}

It is a well-known fact that the block-cut-vertex graph of a connected graph is a tree.  
This observation makes the block-cut-vertex graph, and more generally the block structure of a graph, a useful tool, specifically when studying graphs with connectivity one.  
By definition, each block of a graph $G$ consists of at least two vertices.  
A block $B$ of $G$ is \textit{proper} if $|V(B)|\ge 3$. 
When studying a block decomposition of $G$, the structure of proper blocks is often of interest.  
In particular, at times one might hope that the proper blocks will be pairwise vertex-disjoint.  
In general, however, such an ideal structure is not possible.  However, the general problem of determining conditions that ensure a graph has a 2-proper partition, addressed in one of many possible ways by Theorem \ref{thm1}, can be viewed as a vertex analogue to that of determining when a graph has vertex-disjoint proper blocks. 

\medbreak\noindent\textit{Proof of Theorem~\ref{thm1}.}\quad
We proceed by induction on $n$, with the base cases $n\leq 4$ being trivial.
Thus we may assume that $n\geq 5$.

First suppose that $G$ is disconnected, and let $G_{1},\cdots ,G_{m}$ be the components of $G$.  
For each $1\leq i\leq m$, since
\[
	\delta (G_{i})\geq \delta (G)\geq \sqrt{n}>\sqrt{|V(G_{i})|},
\]
$G_{i}$ has a $2$-proper partition $\P_{i}$ with at most $(|V(G_i)|-1)/\delta(G_{i})~(\leq (|V(G_{i})|-1)/\delta (G))$ parts, by induction.  
Therefore, $\P =\bigcup _{1\leq i\leq m}\P_{i}$ is a $2$-proper partition of $G$ with
\[
	|\P |=\sum _{1\leq i\leq m}|\P _{i}|\leq \sum _{1\leq i\leq m}(|V(G_{i})|-1)/\delta (G)<(n-1)/\delta (G).
\]

Hence we may assume that $G$ is connected.  
If $G$ is $2$-connected, then the trivial partition $\P =\{V(G)\}$ is a desired $2$-proper partition of $G$, so we proceed by supposing that $G$ has at least one cut-vertex.

\begin{claim}
\label{cl2.1}
If $G$ has a block $B$ of order at least $2\delta (G)$, then $G$ has a $2$-proper partition $\P$ with $|\P|\leq (n-1)/\delta (G)$.
\end{claim}
\begin{proof}
It follows that 
\[
	|V(G)-V(B)|\leq n-2\delta (G)\leq n-2\sqrt{n},
\]
and 
\[
	\delta (G-V(B))\geq \delta (G)-1\geq \sqrt{n}-1.
\]
Since $\sqrt{n}-1=\sqrt{n-2\sqrt{n}+1}>\sqrt{n-2\sqrt{n}}$,
\[
	\delta (G-V(B))\geq \sqrt{n}-1>\sqrt{|V(G)-V(B)|}.
\]
Applying the induction hypothesis, $G-V(B)$ has a $2$-proper partition $\P$ with 
\[
	|\P|\leq (n-|V(B)|-1)/\delta (G-V(B))\leq (n-2\delta (G)-1)/(\delta (G)-1).
\]
Since $(n-1)(\delta (G)-1)-(n-\delta (G)-2)\delta (G)=\delta (G)^{2}-n+\delta (G)+1>n-n=0$, $(n-1)/\delta (G)\geq (n-\delta (G)-2)/(\delta (G)-1)$, and hence
\[
	|\P\cup \{V(B)\}|\leq \frac{n-2\delta (G)-1}{\delta (G)-1}+1=\frac{n-\delta (G)-2}{\delta (G)-1}\leq \frac{n-1}{\delta (G)}.
\]
Consequently $\P\cup \{V(B)\}$ is a $2$-proper partition of $G$ with $|\P\cup \{V(B)\}|\leq (n-1)/\delta (G)$.
\end{proof}

By Claim~\ref{cl2.1}, we may assume that every block of $G$ has order at most $2\delta (G)-1$.
Let $\B$ be the set of blocks of $G$.
For each $B\in \B$, let $X_{B}=\{x\in V(B)\mid x$ is not a cut-vertex of $G\}$.
Note that $N_{G}(x)\subseteq V(B)$ for every $x\in X_{B}$.
Let $X=\bigcup _{B\in \B}X_{B}$.
For each vertex $x$ of $G$, let $\B_{x}=\{B\in \B\mid x\in V(B)\}$.
In particular, for each cut-vertex $x$ of $G$ we have $|\B_{x}|\geq 2$.

\begin{claim}
\label{cl2.2}
Let $x$ be a cut-vertex of $G$, and let $C$ be a component of $G-x$.
Then $|V(C)|\geq \delta (G)$.
In particular, every end-block of $G$ has order at least $\delta (G)+1$.
\end{claim}
\begin{proof}
Let $y\in V(C)$.
Note that $d_{C}(y)\geq d_{G}(y)-1\geq \delta (G)-1$.
Since $N_{C}(y)\cup \{y\}\subseteq V(C)$, $(\delta (G)-1)+1\leq d_{C}(y)+1\leq |V(C)|$.
\end{proof}

\begin{claim}
\label{cl2.3}
For each $x\in V(G)$, $|N_{G}(x)\cap X|\geq 2$.
In particular, for a block $B$ of $G$, if $X_{B}\not=\emptyset $, then $|X_{B}|\geq 3$.
\end{claim}
\begin{proof}
Suppose that $|N_{G}(x)\cap X|\leq 1$.  
For each vertex $y\in N_{G}(x)-X$, since $y$ is a cut-vertex of $G$, there exists a component $C_{y}$ of $G-y$ such that $x\not\in V(C_{y})$.
By Claim~\ref{cl2.2}, $|V(C_{y})|\geq \delta (G)$.
Futhermore, for distinct vertices $y_{1},y_{2}\in N_{G}(x)-X$, we have $V(C_{y_{1}})\cap (V(C_{y_{2}})\cup N_{G}(x))=\emptyset$.
Hence
\begin{align*}
	n &\geq |N_{G}(x)\cup \{x\}|+\sum _{y\in N_{G}(x)-X}|V(C_{y})|\\ 
		&\geq (\delta (G)+1)+|N_G(x)-X|\delta(G)\\ 
		&\geq (\delta(G)+1)+(\delta (G)-1)\delta (G)\\
		&=\delta (G)^{2}+1\\ 
		&\geq n+1,
\end{align*} 
which is a contradiction.
\end{proof}

\begin{claim}
\label{cl2.4}
Let $B$ be a block of $G$ with $X_{B}\not=\emptyset $, and let $x\in V(B)$ be a cut-vertex of $G$.
Then there exists a block $C$ of $B-x$ with $X_{B}\subseteq V(C)$.
In particular, if $B$ is an end-block of $G$, then $B-x$ is $2$-connected.
\end{claim}
\begin{proof}
For the moment, we show that any two vertices in $X_{B}$ belong to a common block of $B-x$.
By way of contradiction, we suppose that there are distinct vertices $y_{1},y_{2}\in X_{B}$ such that no block of $B-x$ contains both $y_{1}$ and $y_{2}$.
In particular, $y_{1}y_{2}\not\in E(G)$.
Then $|N_{B-x}(y_{1})\cap N_{B-x}(y_{2})|\leq 1$, and hence $|N_{B-x}(y_{1})\cup N_{B-x}(y_{2})\cup \{y_{1},y_{2}\}|=|N_{B-x}(y_{1})|+|N_{B-x}(y_{2})|-|N_{B-x}(y_{1})\cap N_{B-x}(y_{2})|+2\geq 2(\delta (G)-1)+1$.
It follows that
$$
|V(B)-\{x\}|\geq |N_{B-x}(y_{1})\cup N_{B-x}(y_{2})\cup \{y_{1},y_{2}\}|\geq 2\delta (G)-1,
$$
and hence $|V(B)|\geq 2\delta (G)$, which contradicts the assumption that every block of $G$ has order at most $2\delta (G)-1$.
Thus any two vertices in $X_{B}$ belong to a common block of $B-x$.
This together with the definition of a block implies that a block $C$ of $B-x$ satisfies $X_{B}\subseteq V(C)$.
\end{proof}

Fix an end-block $B_{0}$ of $G$.
Then we can regard the block-cut-vertex graph $T$ of $G$ as a rooted tree with the root $B_{0}$.
For a block $B$ of $G$, let $G(B)$ denote the subgraph which consists of $B$ and the descendant blocks of $B$ with respect to $T$ (i.e., $G(B)$ is the graph formed by the union of all blocks of $G$ contained in the rooted subtree of $T$ with the root $B$).
A $2$-proper partition $\P$ of a subgraph of $G$ is {\it extendable} if $|P|\geq \delta (G)$ for every $P\in \P$.

\begin{claim}
\label{cl2.5}
Let $B^{*}$ be a block of $G$ with $B^{*}\not=B_{0}$, and let $u\in V(B^{*})$ be the parent of $B^{*}$ with respect to $T$.
Then $G(B^{*})-u$ has an extendable $2$-proper partition.
Furthermore, if $X_{B^{*}}\not=\emptyset $, then $G(B^{*})$ has an extendable $2$-proper partition.
\end{claim}
\begin{proof}
We proceed by induction on the height $h$ of the block-cut-vertex graph of $G(B^{*})$ with the root $B^{*}$.
If $h=0$, then $G(B^{*})~(=B^{*})$ is an end-block of $G$, and hence the desired conclusion holds by Claims~\ref{cl2.2} and \ref{cl2.4}.
Thus we may assume that $h\geq 2$ (i.e., $B^{*}$ has a child in $T$).
By the assumption of induction, for $x\in V(B^{*})-(X_{B^{*}}\cup \{u\})$ and $B\in \B_{x}-\{B^{*}\}$, $G(B)-x$ has an extendable $2$-proper partition $\P_{x,B}$.
For each $x\in V(B^{*})-(X_{B^{*}}\cup \{u\})$, let $\P_{x}=\bigcup _{B\in \B_{x}-\{B^{*}\}}\P_{x,B}$ and fix a block $B_{x}\in \B_{x}-\{B^{*}\}$ so that $X_{B_x}$ is not empty, if possible. 

Suppose that $X_{B^{*}}=\emptyset $.
Fix a vertex $x\in V(B^{*})-\{u\}$.
Then by Claim~\ref{cl2.3}, we may assume that $X_{B_{x}}\not=\emptyset $.
By the assumption of induction, $G(B_{x})$ has an extendable $2$-proper partition $\Q_{x}$.
This together with the assumption that $X_{B^{*}}=\emptyset $ implies that $\bigcup _{x\in V(B^{*})-\{u\}}((\P_{x}-\P_{x,B_x})\cup \Q_{x})$ is an extendable $2$-proper partition of $G(B^{*})-u$, as desired.
Thus we may assume that $X_{B^{*}}\not=\emptyset $.

\begin{subclaim}
\label{cl2.5.1}
There exists a block $A$ of $B^{*}-u$ such that
\begin{enumerate}[{\upshape(i)}]
\item
$X_{B^{*}}\subseteq V(A)$,
\item
$|V(A)|\geq \delta (G)$, and
\item
for $x\in V(B^{*})-(V(A)\cup \{u\})$, there exists a block $B'_{x}\in \B_{x}-\{B^{*}\}$ with $X_{B'_{x}}\not=\emptyset $.
\end{enumerate}
\end{subclaim}
\begin{proof}
By Claim~\ref{cl2.4}, there exists a block $A$ of $B^{*}-u$ satisfying (i).
We first show that $A$ satisfies (ii).
Suppose that $|V(A)|\leq \delta (G)-1$.
By the definition of a block, for any $x,x'\in X_{B^{*}}$ with $x\not=x'$, $N_{B^{*}-u}(x)\cap N_{B^{*}-u}(x')\subseteq V(A)$, and so $|(N_{B^{*}-u}(x)-V(A))\cup (N_{B^{*}-u}(x')-V(A))|=|N_{B^{*}-u}(x)-V(A)|+|N_{B^{*}-u}(x')-V(A)|$.
For each $x\in X_{B^{*}}$, since $x\in V(A)-N_{B^{*}-u}(x)$, $|N_{B^{*}-u}(x)-V(A)|\geq \delta (G)-1-(|V(A)|-1)$.
By Claim~\ref{cl2.2}, $|V(G(B_{x}))-\{x\}|\geq \delta (G)$ for every $x\in V(B^{*})-(X_{B^{*}}\cup \{u\})$.
Hence by Claim~\ref{cl2.3},
\begin{align*}
	n	&\geq  \left|(V(B^{*})-\{u\})\cup \left(\bigcup _{x\in V(B^{*})-(X_{B^{*}}\cup \{u\})}\left(V(G(B_{x}))-\{x\}\right)\right)\right| \\
		&= |V(B^{*})-\{u\}|+\sum _{x\in V(B^{*})-(X_{B^{*}}\cup \{u\})}|V(G(B_{x}))-\{x\}| \\
		&\geq  |V(B^{*})-\{u\}|+\delta (G)\left(|V(B^{*})-\{u\}|-|X_{B^{*}}|\right) \\
		&= (\delta (G)+1)|V(B^{*})-\{u\}|-\delta (G)|X_{B^{*}}| \\
		&\geq  (\delta (G)+1) \left|V(A)\cup \left(\bigcup _{x\in X_{B^{*}}}\left(N_{B^{*}-u}(x)-V(A)\right)\right)\right|-\delta (G)|X_{B^{*}}| \\
		&= (\delta (G)+1)\left(|V(A)|+\sum _{x\in X_{B^{*}}}\left|N_{B^{*}-u}(x)-V(A)\right|\right)-\delta (G)|X_{B^{*}}| \\
		&\geq  (\delta (G)+1)\left(|V(A)|+\sum _{x\in X_{B^{*}}}\left(\delta (G)-1-(|V(A)|-1)\right)\right)-\delta (G)|X_{B^{*}}| \\
		&= (\delta (G)+1)(|V(A)|+|X_{B^{*}}|(\delta (G)-|V(A)|))-\delta (G)|X_{B^{*}}| \\
		&= |X_{B^{*}}|\delta (G)^{2}-|V(A)|(\delta (G)+1)(|X_{B^{*}}|-1) \\
		&\geq  |X_{B^{*}}|\delta (G)^{2}-(\delta (G)-1)(\delta (G)+1)(|X_{B^{*}}|-1) \\
		&= \delta (G)^{2}+|X_{B^{*}}|-1 \\
		&\geq  n+3-1,
\end{align*}
which is a contradiction.
Thus $|V(A)|\geq \delta (G)$.

We next check that $A$ satisfies (iii).
Let $x\in V(B^{*})-(V(A)\cup \{u\})$.
Since $A$ is a block of $B^{*}-u$ and satisfies (i), $|N_{G}(x)\cap X_{B^{*}}|\leq 1$.
This together with Claim~\ref{cl2.3} implies that there exists a block $B'_{x}\in \B_{x}-\{B^{*}\}$ with $X_{B'_{x}}\not=\emptyset $.
\end{proof}

Let $A$ and $B'_{x}\in \B_{x}-\{B^{*}\}~(x\in V(B^{*})-(V(A)\cup \{u\}))$ be as in Subclaim~\ref{cl2.5.1}.
By the assumption of induction, for $x\in V(B^{*})-(V(A)\cup \{u\})$, $G(B'_{x})$ has an extendable $2$-proper partition $\Q'_{x}$.
Then
$$
\{V(A)\}\cup \left(\bigcup _{x\in V(B^{*})-(V(A)\cup \{u\})}((\P_{x}-\P_{x,B'_{x}})\cup \Q'_{x})\right)\cup \left(\bigcup _{x\in V(A)-X_{B^{*}}}\P_{x}\right)
$$
is an extendable $2$-proper partition of $G(B^{*})-u$.

Since $N_{G}(x)\cup \{x\}\subseteq V(B^{*})$ for $x\in X_{B^{*}}$, $|V(B^{*})|\geq \delta (G)+1$, and hence $\{V(B^{*})\}\cup (\bigcup _{x\in V(B^{*})-(X_{B^{*}}\cup \{u\})}\P_{x})$ is an extendable $2$-proper partition of $G(B^{*})$.
\end{proof}

By Claim~\ref{cl2.5}, $G-V(B_{0})$ has an extendable $2$-proper partition $\P_{0}$.
Hence $\P=\{V(B_{0})\}\cup \P_{0}$ is a $2$-proper partition of $G$.
Furthermore, since $|V(B_{0})|\geq \delta (G)+1$ by Claim~\ref{cl2.2}, $n=\sum _{P\in \P}|P|=|V(B_{0})|+\sum _{P\in \P_{0}}|P|\geq (\delta (G)+1)+(|\P|-1)\delta (G)=|\P|\delta (G)+1$, and hence $|\P|\leq (n-1)/\delta (G)$.

This completes the proof of Theorem~\ref{thm1}.
\qed


\section{$k$-Proper Partitions}\label{sec3}

Let $e(k,n)$ be the maximum number of edges in a graph of order $n$ with no $k$-connected subgraph.
Define $d(k)$ to be $$\sup\left\{ \frac{2e(k,n)+2}{n} : n > k\right\}$$ and
$$\gamma = \sup\{ d(k)/(k-1) : k \geq 2 \}.$$

Recall that the average degree of a graph $G$ of order $n$ with $e(G)$ edges is $\frac{2e(G)}{n}$.  This leads to the following useful observation.   

\begin{obs}\label{obs:gamma}
If $G$ is a graph with average degree at least $\gamma(k-1)$, then $G$ contains a $k$-connected subgraph.
\end{obs}

In \cite{Mader}, Mader proved that $3 \leq \gamma \leq 4$ and constructed a graph of order $n$ with $\left(\frac{3}{2}k-2\right)(n-k+1)$ edges and without $k$-connected subgraphs.  This led him to make the following conjecture.  

\begin{con}\label{conj:mader}
	If $k\ge 2$, then $e(k,n) \leq \left(\frac{3}{2}k-2\right)(n-k+1)$.  Consequently, $d(k) \leq 3(k-1)$ and $\gamma = 3$.
\end{con}

Note that Conjecture \ref{conj:mader} holds when $k=2$, as it is straightforward to show that $e(2,n)=n-1$.  The most significant progress towards Conjecture \ref{conj:mader} is due to Yuster \cite{Yuster}.

\begin{thm}\label{thm:Yuster}
	If $n \geq \frac{9}{4}(k-1)$, then $e(k,n) \leq \frac{193}{120}(k-1)(n-k+1)$.
\end{thm}

Note that Theorem \ref{thm:Yuster} requires $n \geq \frac{9}{4}(k-1)$, which means that we cannot immediately obtain a bound on $\gamma$.
The following corollary, however, shows that we can use this result in a manner similar to Observation \ref{obs:gamma}.

\begin{cor}\label{cor:Yuster}
Let $G$ be a graph of order $n$ with average degree $\overline{d}$.
Then $G$ contains a $\lfloor\frac{60\overline{d}}{193}\rfloor$-connected subgraph.
\end{cor}

\begin{proof}
Let $k = \lfloor\frac{60\overline{d}}{193}\rfloor$ and suppose that $G$ does not contain a $k$-connected subgraph.  If $n \geq \frac{9}{4}(k-1)$, then Theorem~\ref{thm:Yuster} implies
\[
\frac{1}{2}\overline{d} n = e(G) \leq \frac{193}{120}(k-1)(n-k+1) < \frac{193}{120} \left(\frac{60}{193}\overline{d}\right)n = \frac{1}{2} \overline{d} n.
\]
Thus, assume that $n < \frac{9}{4}(k-1)$.
This implies that
\[
n < \frac{9}{4}(k-1) < \frac{9}{4}\frac{60}{193}\overline{d} < \frac{7}{10}\Delta(G),
\]
a contradiction.
\end{proof}

Finally, prior to proving our main result, we require the following simple lemma, which we present without proof.  

\begin{lem}\label{lemma:kconn}
If $G$ is a graph of order $n\ge k+1$ such that $\delta(G)\ge\frac{n+k-2}{2}$, then $G$ is $k$-connected.  
\end{lem}

We prove the following general result, and then show that we may adapt the proof to improve Theorem \ref{thm2}.

\begin{thm}\label{thm:fewerparts}\label{thm3}
Let $k \geq 2$ and $c \geq \frac{11}{3}$.
If $G$ is a graph of order $n$ with minimum degree $\delta$ with $\delta \geq \sqrt{c\gamma(k-1) n}$, then $G$ has a $k$-proper partition into at most $\left\lfloor\frac{c\gamma n}{\delta}\right\rfloor$ parts.
\end{thm}

\begin{proof}
Since $n > \delta \geq \sqrt{c\gamma(k-1)n}$, we have $n^2 > c\gamma(k-1)n$ and hence $n > c\gamma(k-1) \geq 11(k-1)$.
Therefore, by Lemma \ref{lemma:kconn}, it follows that 
\[
	\delta < \frac{n+k-2}{2} < \frac{n+(k-1)}{2} \leq \frac{n+\frac{1}{11}n}{2} \leq \frac{6}{11}n.
\]

Let $G_0 = G$, $\delta_0 = \delta$, and $n_0 = |V(G)|$.
We will build a sequence of graphs $G_i$ of order $n_i$ and minimum degree $\delta_i$ by selecting a $k$-connected subgraph $H_i$ of largest order from $G_i$ and assigning $G_{i+1} = G_i - V(H_i)$.
This process terminates when either $G_{i}$ is $k$-connected or $G_i$ does not contain a $k$-connected subgraph.
We claim the process terminates when $G_i$ is $k$-connected and $H_i = G_i$.

By Observation \ref{obs:gamma}, $G_i$ contains a $( \lfloor\frac{\delta_i}{\gamma}\rfloor+1 )$-connected subgraph $H_i$.
If $\frac{\delta_i}{\gamma} \geq k-1$, then $H_i$ is $k$-connected and has order at least $\lfloor\frac{\delta_i}{\gamma}\rfloor+1>\frac{\delta_i}{\gamma}$.  Since $H_i$ is a maximal $k$-connected subgraph in $G_i$, every vertex $v \in V(G_i) \setminus V(H_i)$ has at most $k-1$ edges to $H_i$ by a simple consequence of Menger's Theorem.
Therefore, we have $$\delta_{i+1} \geq \delta_i - (k-1)$$ and $$n_{i+1} =n_i-|H_i|<n_i - \delta_i/\gamma.$$
This gives us the estimates on $\delta_i$ and $n_i$ of
\[
	\delta_i \geq \delta - i(k-1),
\]
and
\begin{align*}
	n_i &\leq n - \sum_{j=0}^{i-1} \delta_j/\gamma
		\leq n  - \frac{1}{\gamma}\sum_{j=0}^{i-1}\left[\delta - j(k-1)\right]
		= n  - \frac{1}{\gamma}\left[ i\delta - (k-1)\binom{i}{2}\right].  
\end{align*}

Let $t = \left\lceil\frac{c\gamma n}{\delta}-4\right\rceil = \frac{c\gamma n}{\delta} - (4-x)$, where $x\in[0,1)$.
We claim that the process terminates with a $k$-proper partition at or before the $(t+1)^{\text{st}}$ iteration (that is, at or before the point of selecting a $k$-connected subgraph from $G_t$).  First, we have
\[
	\delta_{t-1} \geq \delta - (t-1)(k-1)
	> \delta - \left(\frac{c\gamma n}{\delta} - 4\right)(k-1)
	= \delta - \frac{c\gamma(k-1)n}{\delta} + 4(k-1).
\]
Note that $\delta^2 \geq  c\gamma(k-1) n$ and hence $\delta-\frac{c\gamma (k-1) n}{\delta} \geq 0$.
Therefore, 
\[\delta_{t-1} > 4(k-1) \geq \gamma(k-1) \quad\text{and}\quad \delta_t \geq 3(k-1).\]
As the bound on $\delta_i$ is a decreasing function of $i$, we have $\delta_i>4(k-1)$ for all $0\le i\le t-1$.  Thus each $G_i$ with $i<t$ contains a $k$-connected subgraph.
Next, consider $n_t$.
\begin{align*}
n_t &\leq n - \frac{1}{\gamma}\left[ t\delta - (k-1)\binom{t}{2}\right] \\
	&= n - \frac{1}{\gamma}\left[  c\gamma n - (4-x)\delta - \frac{1}{2}(k-1)\left( \frac{c\gamma n}{\delta} - (4-x)\right)\left( \frac{c\gamma n}{\delta} - (5-x)\right)\right] \\
	&= n - \frac{1}{\gamma}\left[  c\gamma n- (4-x)\delta - \frac{c^2\gamma^2(k-1)}{2}\frac{n^2}{\delta^2} + \frac{c\gamma(9-2x)(k-1)}{2}\frac{n}{\delta} - \frac{1}{2}(k-1)(4-x)(5-x)\right]\\
	&=\frac{1}{\delta^2}\left[ \frac{4-x}{\gamma}\delta^3 + \frac{c^2\gamma(k-1)}{2}n^2 -  (c-1){n\delta^2}  \right] + \frac{(4-x)(5-x)}{2\gamma}(k-1)  - \frac{c(9-2x)(k-1)}{2}\frac{n}{\delta}.
\end{align*}
We have $\delta^2 \geq c\gamma(k-1)n$ and $(c-1)^2 \geq c^2/2$, so 
\begin{align*}
\frac{(c-1)}{c}((c-1)n\delta^2)
	\geq (c-1)^2\gamma(k-1)n^2
	\geq  \frac{c^2}{2}\gamma(k-1)n^2.
\end{align*}
Also,  we have $n > \frac{11}{6}\delta$, and $\frac{c-1}{c} \geq \frac{8}{11}$, hence
\begin{align*}
\frac{1}{c}((c-1)n\delta^2)
	> \frac{8}{11}\cdot \frac{11}{6}\delta^3
	= \frac{4}{3}\delta^3
	\geq \frac{4-x}{\gamma}\delta^3.
\end{align*}
Summing these inequalities, we get that 
\[
\left[ \frac{4-x}{\gamma}\delta^3 + \frac{c^2\gamma(k-1)}{2}n^2 - (c-1){n\delta^2}\right] < 0
\]
and hence $n_t <\frac{(4-x)(5-x)}{2\gamma}(k-1)  \leq \frac{20}{2\gamma}(k-1) \leq \frac{10}{3}(k-1)$.
However, $\delta_t \geq 3(k-1)$, so if the process has not terminated prior to the $(t+1)^{\text{st}}$ iteration, $G_t$ is $k$-connected by Lemma \ref{lemma:kconn}.
\end{proof}

Theorem \ref{thm:fewerparts} immediately yields the following.  

\begin{cor}
Suppose Conjecture~\ref{conj:mader} holds.
We then see that if $G$ is a graph with minimum degree $\delta$ where $\delta \geq \sqrt{11(k-1)n}$, then $G$ has a $k$-proper partition into at most $\frac{11n}{\delta}$ parts.
\end{cor}

We are now ready to prove Theorem \ref{thm2}.

\proof Observe that the proof of Theorem~\ref{thm:fewerparts} holds at every step when substituting $\gamma = \frac{193}{60}$ by using Corollary~\ref{cor:Yuster} to imply that $G_i$ contains a $\lfloor\frac{60\delta_i}{193}\rfloor$-connected subgraph. 
Finally, note that $\left(\frac{11}{3}\right)\frac{193}{60}=\frac{2123}{180}$.\qed

\section{Application:  Edit Distance to the Family of $k$-connected Graphs}
Define the \textit{edit distance} between two graphs $G$ and $H$ to be the number of edges one must add or remove to obtain $H$ from $G$ (edit distance was introduced independently in \cite{a,akm,r}). More generally, the edit distance between a graph $G$ and a set of graphs $\mathcal{G}$ is the minimum edit distance between $G$ and some graph in $\mathcal{G}$.  

Utilizing Theorem~\ref{thm2} and observing that $2123/180=11.79\overline{4}<11.8$  we obtain the following corollary, which is a refinement of Corollary 11 in \cite{FMW} for large enough $k$.  

\begin{cor}
Let $k\ge 2$ and let $G$ be a graph of order $n$.  If $\delta(G) \geq \sqrt{11.8(k-1)n}$, then the edit distance between $G$ and the family of $k$-connected graphs of order $n$ is at most $\frac{11.8kn}{\delta(G)}-k< k(4\sqrt{n}-1)$.   
\end{cor}

\begin{proof}
Let $H_1,\ldots ,H_l$ be the $k$-connected subgraphs of the $k$-proper partition of $G$ guaranteed by Theorem~\ref{thm2}; note that $l\le\frac{11.8n}{\delta(G)}$. 
For each $i\in \{1,\ldots, l-1\}$, it is possible to produce a matching of size $k$ between $H_i$ and $H_{i+1}$ by adding at most $k$ 
edges between $H_i$ and $H_{i+1}$. Thus, adding at most $k\left(\frac{11.8n}{\delta(G)}\right)$ edges yields a $k$-connected graph. \end{proof}

\section{Conclusion}

We note here that it is possible to slightly improve the degree conditions in Theorems \ref{thm2} and \ref{thm3} at the expense of the number of parts in the partition.  In particular, a greedy approach identical to that used to prove Theorem \ref{thm3} can be used to prove the following.  

\begin{thm}\label{thm:main}
Let $k \geq 2$, $c_k \geq \frac{k-1}{k} \cdot 2\gamma$, and $p = \sqrt{ \frac{c_k n}{k} }$.
If $G$ is a graph of order $n$ with $\delta(G) \geq kp = \sqrt{c_k k n}$, then $G$ has a $k$-proper partition into at most $\frac{k}{k-1}p$ parts.
\end{thm}

This gives rise to the following, which improves on the degree condition in Theorem \ref{thm2}.  

\begin{thm}\label{partition_improve}
If $G$ is a graph of order $n$ with minimum degree $$\delta(G)\geq kp = \sqrt{\frac{193}{30}(k-1)n},$$ then $G$ has a $k$-proper partition into at most $\frac{k}{k-1}p$ parts.
\end{thm}

\subsection*{Acknowledgments}

The second author would like to acknowledge the generous support of the Simons Foundation.  His research is supported in part by Simons Foundation Collaboration Grant \#206692.   The third author would like to thank the laboratory LRI of the University Paris South and Digiteo foundation for their generous hospitality.  He was able to carry out part of this research during his visit there. Also, the third author's research is supported by the Japan Society for the Promotion of Science Grant-in-Aid for Young Scientists (B) (20740095).

\end{document}